\documentclass[11pt,a4paper]{article}
\usepackage{a4wide}
\usepackage{amsmath}
\usepackage{amssymb}
\usepackage{algorithm,algorithmic}
\usepackage{graphicx}
\usepackage[framed,amsmath,amsthm,thmmarks]{ntheorem}
\usepackage{framed}
\usepackage{color}

\allowdisplaybreaks

  \makeatletter
    
    \@addtoreset{equation}{section}
  \makeatother

  \makeatletter
    
    \@addtoreset{algorithm}{section}
  \makeatother

  \makeatletter
    
    \@addtoreset{figure}{section}
  \makeatother

  \makeatletter
    
    \@addtoreset{table}{section}
  \makeatother

\definecolor{bl}{gray}{0.9}

\def\FR{\mathop{\rm FR}\nolimits}
\def\PR{\mathop{\rm PR}\nolimits}

\def\St{\mathop{\rm St}\nolimits}
\def\tr{\mathop{\rm tr}\nolimits}
\def\diag{\mathop{\rm diag}\nolimits}

\def\grad{\mathop{\rm grad}\nolimits}

\def\D{{\rm D}}

\def\id{\mathop{\rm id}\nolimits}

\def\qf{\mathop{\rm qf}\nolimits}

\providecommand{\abs}[1]{\lvert#1\rvert}
\providecommand{\norm}[1]{\lVert#1\rVert}

\newtheorem{Def}{Definition}[section]
\newtheorem{Thm}{Theorem}[section]
\newtheorem{Prop}{Proposition}[section]
\newtheorem{Lemma}{Lemma}[section]
\newtheorem{Proof}{Proof}[section]
\newtheorem{Prob}{Problem}[section]

\newtheorem{predfn}{Definition}[section]
\newtheorem{prerem}[predfn]{Remark}
\newenvironment{Remark}{\begin{prerem}\itshape}{\end{prerem}}

\newframedtheorem{mydef}[Def]{Definition}
\newframedtheorem{mythm}[Thm]{Theorem}
\newframedtheorem{myprop}[Prop]{Proposition}
\newframedtheorem{mylemma}[Lemma]{Lemma}
\newframedtheorem{myproof}[Proof]{Proof}
\newframedtheorem{myprob}[Prob]{Problem}

\title{A new, globally convergent \\ Riemannian conjugate gradient method}
\author{Hiroyuki Sato\thanks{{\tt hsato@amp.i.kyoto-u.ac.jp}}\quad and\quad Toshihiro Iwai\\
Department of Applied Mathematics and Physics\\ Kyoto University, Kyoto 606-8501, Japan}
\date{\today}

\begin{document}
\maketitle
\begin{abstract}
This article deals with the conjugate gradient method on a Riemannian manifold with interest in global convergence analysis.
The existing conjugate gradient algorithms on a manifold endowed with a vector transport need the assumption that the vector transport does not increase the norm of tangent vectors, in order to confirm that generated sequences have a global convergence property.
In this article, the notion of a scaled vector transport is introduced to improve the algorithm so that the generated sequences may have a global convergence property under a relaxed assumption.
In the proposed algorithm, the transported vector is rescaled in case its norm has increased during the transport.
The global convergence is theoretically proved and numerically observed with examples.
In fact, numerical experiments show that there exist minimization problems for which the existing algorithm generates divergent sequences, but the proposed algorithm generates convergent sequences.\bigskip
\end{abstract}


\pagenumbering{arabic}
\section{Introduction}
The conjugate gradient method was first developed by Hestenes and Stiefel as a tool for solving the linear equation $Ax=b$, where $A$ is an $n\times n$ positive definite matrix \cite{hestenes}.
The strategy of the linear conjugate gradient method is to minimize the quadratic function $x^TAx/2-b^Tx$ of $x$ in the successive search directions which are generated in such a manner that those directions are mutually conjugate with respect to $A$ and eventually span the whole $\mathbb R^n$.
As this method is generalized to be applicable to functions which are not restricted to those quadratic in $x$, the conjugate gradient method in its original form
is particularly called the linear conjugate gradient method.

According to a nonlinear conjugate gradient method for minimizing a smooth function $f$ which is not necessarily quadratic,
the search direction $\eta_k$ is determined by
\begin{equation}
\eta_k=-\grad f(x_k)+\beta_k\eta_{k-1},\label{CGR}
\end{equation}
where $\beta_k$ is a parameter to be defined suitably.
Fletcher and Reeves \cite{fletcher} proposed to define $\beta_k$ by $\beta_k:=\norm{\grad f(x_k)}^2/\norm{\grad f(x_{k-1})}^2$ (see \cite{nocedal} for another way to determine $\beta_k$).

On the other hand, iterative optimization methods on $\mathbb R^n$ have been developed so as to be applicable on Riemannian manifolds \cite{absil,edel}.
Those generalized methods are called Riemannian optimization methods, which provide procedures for minimizing objective functions defined on a Riemannian manifold $M$.
In a Riemannian optimization method, the usual line search should be replaced \cite{absil}, as the concept of a line is generalized on a Riemannian manifold.
Absil, Mahony, and Sepulchre proposed to use a retraction map to perform a search on a curve on $M$ in place of the line search.
As for the conjugate gradient method, Smith provided in \cite{smith} a conjugate gradient method on $M$ along with other optimization algorithms on $M$.
The difficulty we encounter in generalizing the conjugate gradient method to that on a manifold is that Eq.~\eqref{CGR} makes no longer sense.
This is because $\grad f(x_k)$ and $\eta_{k-1}$ belong to tangent spaces at different points on $M$ in general, so that they cannot be added.
Smith proposed to use the parallel translation along the geodesic at each iteration in order to make possible the addition of two tangent vectors and thereby to extend the iteration procedure \eqref{CGR}.
However, using the parallel translation on $M$ is not computationally effective in general.
A way to perform the conjugate gradient method on $M$ in an efficient manner is to use a vector transport \cite{absil}.
The global convergence in the conjugate gradient method with a vector transport on $M$ has been recently discussed by Ring and Wirth \cite{ring}.
They proved the global convergence under the condition that the vector transport in use does not increase the norm of the search direction vector.
On the contrary, the present article provides numerical evidence to show that if the assumption is not satisfied, the conjugate gradient method with a general vector transport may fail to generate a globally converging series.
In order to relax the assumption in \cite{ring}, the notion of a ``scaled" vector transport is introduced in this article and a new conjugate gradient algorithm is proposed with only a mild computational overhead per iteration. 

The organization of this paper is as follows:
The scaled vector transport is introduced in Section \ref{cg} after a brief review of some useful existing concepts.
How to compute the step size is also discussed in this section.
In Section \ref{cgr}, a brief review is made of the conjugate gradient method on a Riemannian manifold $M$,
and then a new algorithm is proposed, in which the scaled vector transport is applied
only if the vector transport increases the norm of the previous search direction.
In Section \ref{analysis}, the global convergence for the proposed algorithm is proved in a manner similar to the usual one performed on $\mathbb R^n$, where the scaled vector transport used on a fitting occasion makes a generated sequence into a globally convergent one.
Section \ref{experiments} provides numerical experiments on simple problems which the existing algorithm cannot solve efficiently but the proposed algorithm can do.
The numerical experiments show why the present algorithm can generate convergent sequences.
Section \ref{conclusion} includes concluding remarks.
It is shown in Appendix \ref{app} that the Lipschitzian condition referred to in Subsection \ref{zousec} is satisfied for some practical Riemannian optimization problems.

\section{Setup for Riemannian optimization}\label{cg}
\subsection{Retraction}\label{general}
An unconstrained optimization problem on a Riemannian manifold $M$ is described as follows:
\begin{Prob}\label{general_prob}
\begin{align}
{\rm minimize} \,\,\,\,\,& f(x),\\
{\rm subject\,\,to} \,\,\,\,\,& x\in M.
\end{align}
\end{Prob}
If $M$ is the Euclidean space $\mathbb R^n$, the line search is performed with the updating formula
\begin{equation}
x_{k+1}=x_k+\alpha_k\eta_k,\label{linesearch}
\end{equation}
where $x_k, x_{k+1}\in\mathbb R^n$ are a current point and an unknown next point, respectively, and where $\eta_k\in\mathbb R^n$ and $\alpha_k>0$ are a search direction at $x_k$ and a step size, respectively.
However, the line search \eqref{linesearch} does not make sense on a general manifold $M$.
In order to generalize the line search \eqref{linesearch} on $\mathbb R^n$ to that on $M$, the search direction $\eta_k$ should be taken as a tangent vector in $T_{x_k}M$,
and the addition in Eq.~\eqref{linesearch} should be replaced by another suitable operation.
A natural alternative to the line search is a search along the geodesic emanating from $x_k$ in the direction of $\eta_k$,
but the geodesic will cause computational difficulty 
except for a few particular manifolds where the geodesics admit a tractable closed-form expression.
A computationally efficient way is to use the following retraction map introduced in \cite{absil}.

\begin{Def}\label{retractiondef}
Let $M$ and $TM$ be a manifold and the tangent bundle of $M$, respectively.
Let $R:TM\to M$ be a smooth map and $R_x$ the restriction of $R$ to $T_xM$.
The $R$ is called a retraction on $M$, if it has the following properties:
\begin{enumerate}
\item $R_x(0_x)=x$, where $0_x$ denotes the zero element of $T_xM$.\label{ret1}
\item With the canonical identification $T_{0_x}T_xM\simeq T_xM$, $R_x$ satisfies
\begin{equation}
{\rm D}R_x(0_x)=\id_{T_xM},
\end{equation}
where $\D R_x(0_x)$ denotes the derivative of $R_x$ at $0_x$, and $\id_{T_xM}$ the identity map on $T_xM$.\label{ret2}
\end{enumerate}
\end{Def}

As is easily seen, the exponential map on $M$ is a typical example of a retraction.
If we can find a computationally preferable retraction, we can perform an optimization procedure as follows:
\begin{algorithm}[H]
\caption{The general framework of optimization methods for Problem \ref{general_prob} on a Riemannian manifold $M$}\label{algorithm1}
\begin{algorithmic}[1]
\STATE Choose an initial point $x_0\in M$.
\FOR{$k=0,1,2,\ldots$}
\STATE Compute the search direction $\eta_k\in T_{x_k}M$ and the step size $\alpha_k>0$.
\STATE Compute the next iterate by $x_{k+1}:=R_{x_k}(\alpha_k\eta_k)$, where $R$ is a retraction on $M$.
\ENDFOR
\end{algorithmic}
\end{algorithm}
The choice of a search direction and a step size characterizes the individual optimization method.
We proceed to the vector transport in search for computationally efficient conjugate gradient methods.

\subsection{Vector transport and scaled vector transport}\label{vectra}
In a (nonlinear) conjugate gradient method on the Euclidean space $\mathbb R^n$, the search directions $\eta_k$ are chosen to be
\begin{equation}
\eta_k=-\grad f(x_k)+\beta_{k}\eta_{k-1}, \qquad k\ge 0,\label{cg_euclid}
\end{equation}
where $\beta_0=0$, and where $\beta_k$ with $k\ge 1$ are determined in several possible manners.
For example, $\beta_k$ are determined by
\begin{equation}
\beta_{k}^{\FR}=\frac{\grad f(x_{k})^T\grad f(x_{k})}{\grad f(x_{k-1})^T\grad f(x_{k-1})},
\end{equation}
or
\begin{equation}
\beta_{k}^{\PR}=\frac{\grad f(x_{k})^T\left(\grad f(x_{k})-\grad f(x_{k-1})\right)}{\grad f(x_{k-1})^T\grad f(x_{k-1})},\label{PR_euclid}
\end{equation}
where FR and PR are abbreviations of Fletcher-Reeves and Polak-Ribi\`{e}re, respectively \cite{nocedal}.

However, if $\mathbb R^n$ is replaced by a Riemannian manifold $M$, $\grad f(x_k)\in T_{x_k}M$ and $\eta_{k-1}\in T_{x_{k-1}}M$ belong to different tangent spaces, so that $-\grad f(x_k)+\beta_{k}\eta_{k-1}$ in Eq.~\eqref{cg_euclid} does not make sense.
The quantity $\grad f(x_{k})-\grad f(x_{k-1})$ in Eq.~\eqref{PR_euclid} makes no sense on $M$ either.
In order to modify the vector addition in Eqs.~\eqref{cg_euclid} and \eqref{PR_euclid} into a suitable operation on $M$,
Smith proposed to use the parallel translation of tangent vectors along a geodesic \cite{smith}.
However, no computationally efficient formula is known for the parallel translation along a geodesic even for the Stiefel manifold
except when it reduces to the sphere or the orthogonal group.
Absil {\it et al.} \cite{absil} proposed the notion of a vector transport as an alternative to the parallel translation.
The vector transport is a generalization of the parallel translation and can enhance computational efficiency of algorithms, if defined suitably.

In this paper, we focus on the differentiated retraction $\mathcal{T}^R$ as a vector transport, which is defined to be
\begin{equation}
\mathcal{T}^R_{\eta_x}(\xi_x):=\D R_x(\eta_x)[\xi_x],\qquad \eta_x,\xi_x\in T_xM,\label{dr}
\end{equation}
where $R$ is a retraction on $M$.
We here note that $\mathcal{T}^R$ satisfies the conditions in the definition of a vector transport, as is easily verified \cite{absil}.

In what follows, we assume that $M$ is a Riemannian manifold and denote the Riemannian metric evaluated at $x\in M$ by $\langle\cdot,\cdot\rangle_x$.
The norm of a tangent vector $\xi_x\in T_x M$ evaluated at $x\in M$ is defined to be $\norm{\xi_x}_x=\sqrt{\langle\xi_x,\xi_x\rangle}$.
We here have to note that though the parallel translation is an isometry, a vector transport is not required to preserve the norm of vectors in general.
The differentiated retraction $\mathcal{T}^R$ is not always an isometry either.
In analysing the convergence for the conjugate gradient method later, it will be crucial whether the vector transport $\mathcal{T}^R$ increases the norm of vectors or not.
In order to prevent the vector transport $\mathcal{T}^R$ from increasing the norm of vectors,
we define the scaled vector transport $\mathcal{T}^0:TM\oplus TM\to TM$ associated with $\mathcal{T}^R$ as follows:
\begin{Def}\label{t0def}
Let $R$ be a retraction on a Riemannian manifold $M$.
Let $\mathcal{T}^R$ be a vector transport defined by \eqref{dr} with respect to $R$.
The scaled vector transport $\mathcal{T}^0$ associated with $\mathcal{T}^R$ is defined as
\begin{equation}
\mathcal{T}^0_{\eta_x}(\xi_x)=\frac{\norm{\xi_x}_x}{\norm{\mathcal{T}^R_{\eta_x}(\xi_x)}_{R_x(\eta_x)}}\mathcal{T}^R_{\eta_x}(\xi_x),\qquad \eta_x,\xi_x\in T_xM.\label{scaled}
\end{equation}
\end{Def}

The scaled vector transport $\mathcal{T}^0$ thus defined is no longer a vector transport since it is not linear.
However, $\mathcal{T}^0$ satisfies
\begin{equation}
\norm{\mathcal{T}^0_{\eta_x}(\xi_x)}_{R_x(\eta_x)}=\norm{\xi_x}_x,\qquad \eta_x,\xi_x\in T_xM,
\end{equation}
which is a key property for the global convergence of the algorithm we will propose.

\subsection{Strong Wolfe conditions}\label{Wolfe_Subsec}
In computing the step size $\alpha_k$ in the conjugate gradient method on $\mathbb R^n$, the strong Wolfe conditions are often used \cite{nocedal}, which require $\alpha_k$ to satisfy
\begin{eqnarray}
f(x_k+\alpha_k\eta_k)\le f(x_k)+c_1\alpha_k\grad f(x_k)^T\eta_k,\label{wolfe1}\\
\abs{\grad f\left(x_k+\alpha_k\eta_k\right)^T\eta_k}\le c_2\abs{\grad f(x_k)^T\eta_k},\label{wolfe2}
\end{eqnarray}
with $0<c_1<c_2<1$.
In particular, $c_1$ and $c_2$ are often taken so as to satisfy $0<c_1<c_2<1/2$ in the conjugate gradient method.
In order to extend the strong Wolfe conditions on $\mathbb R^n$ to those on $M$, we start by reviewing the strong Wolfe conditions \eqref{wolfe1} and \eqref{wolfe2}.
For a current point $x_k$ and a search direction $\eta_k$, one performs a line search for the function defined by
\begin{equation}
\phi(\alpha)=f(x_k+\alpha\eta_k),\qquad \alpha>0.\label{phialpha}
\end{equation}
Requiring $\alpha_k$ to give a sufficient decrease in the value of $f$, one imposes the condition
\begin{equation}
\phi(\alpha_k)\le \phi(0)+c_1\alpha_k \phi'(0),\label{wolfe01}
\end{equation}
which yields \eqref{wolfe1}.
In order to prevent $\alpha_k$ from being excessively short, the $\alpha_k$ is required to satisfy
\begin{equation}
\abs{\phi'(\alpha_k)}\le c_2\abs{\phi'(0)},\label{wolfe02}
\end{equation}
which implies \eqref{wolfe2}.

In order to generalize the strong Wolfe conditions to those on $M$, we define a function $\phi$ on $M$, in an analogous manner to \eqref{phialpha}, to be
\begin{equation}
\phi(\alpha)=f\left(R_{x_k}(\alpha\eta_k)\right),\qquad \alpha>0,\label{phi}
\end{equation}
where $R$ is a retraction on $M$.
The conditions \eqref{wolfe01} and \eqref{wolfe02} applied to \eqref{phi} give rise to
\begin{equation}
f\left(R_{x_k}(\alpha_k\eta_k)\right)\le f(x_k)+c_1\alpha_k\langle\grad f(x_k),\eta_k\rangle_{x_k},\label{wolfem1}
\end{equation}
\begin{equation}
\abs{\langle\grad f\left(R_{x_k}(\alpha_k\eta_k)\right),\D R_{x_k}\left(\alpha_k\eta_k\right)[\eta_k]\rangle_{R_{x_k}(\alpha_k\eta_k)}}\le c_2\abs{\langle\grad f(x_k),\eta_k\rangle_{x_k}},\label{wolfem2}
\end{equation}
respectively, where $0<c_1<c_2<1$.
We call the conditions \eqref{wolfem1} and \eqref{wolfem2} the strong Wolfe conditions.
The existence of a step size satisfying \eqref{wolfem1} and \eqref{wolfem2} can be shown by an almost verbatim repetition of that for the strong Wolfe conditions on $\mathbb R^n$ (see \cite{nocedal}).

\begin{Prop}\label{wolfeprop}
Let $M$ be a Riemannian manifold with a retraction $R$.
If a smooth objective function $f$ on $M$ is bounded below on $\left\{R_{x_k}(\alpha\eta_k)|\alpha>0\right\}$ for $x_k\in M$ and for a descent direction $\eta_k\in T_{x_k} M$,
and if constants $c_1$ and $c_2$ satisfy $0<c_1<c_2<1$, then there exists a step size $\alpha_k$ which satisfies the strong Wolfe conditions \eqref{wolfem1} and \eqref{wolfem2}.
\end{Prop}
We note that the strong Wolfe conditions \eqref{wolfem1} and \eqref{wolfem2} together with the existence of a step size satisfying them are also discussed in \cite{ring}.

We now look into the second condition \eqref{wolfem2}.
If we introduce a vector transport $\mathcal{T}^R$ as the differentiated retraction given by \eqref{dr},
then Eq.~\eqref{wolfem2} can be expressed as
\begin{equation}
\abs{\langle\grad f\left(R_{x_k}(\alpha_k\eta_k)\right),\mathcal{T}^R_{\alpha_k\eta_k}(\eta_k)\rangle_{R_{x_k}(\alpha_k\eta_k)}}\le c_2\abs{\langle\grad f(x_k),\eta_k\rangle_{x_k}}.\label{wolfem3}
\end{equation}
An idea for further generalization of this condition to that in an algorithm with a general vector transport $\mathcal{T}$ is to replace \eqref{wolfem3} by
\begin{equation}
\abs{\langle\grad f\left(R_{x_k}(\alpha_k\eta_k)\right),\mathcal{T}_{\alpha_k\eta_k}(\eta_k)\rangle_{R_{x_k}(\alpha_k\eta_k)}}\le c_2\abs{\langle\grad f(x_k),\eta_k\rangle_{x_k}}.\label{wolfem4}
\end{equation}
However, if $\mathcal{T}\neq \mathcal{T}^R$, the existence of a step size satisfying both \eqref{wolfem1} and \eqref{wolfem4} is unclear in general.
In view of this, the differentiated retraction $\mathcal{T}^R$ is considered to be a natural choice of a vector transport $\mathcal{T}$,
for which a step size satisfying \eqref{wolfem1} and \eqref{wolfem4} is shown to exist.
In what follows, we use the differentiated retraction $\mathcal{T}^R$ and the scaled one $\mathcal{T}^0$.

\section{A new conjugate gradient method on a Riemannian manifold}\label{cgr}
If a Riemannian manifold $M$ is given a retraction $R$ and the corresponding vector transport $\mathcal{T}^R$,
a standard Fletcher-Reeves type conjugate gradient method on $M$ is described as follows \cite{absil, ring}:
\begin{algorithm}[H]
\caption{A standard Fletcher-Reeves type conjugate gradient method for Problem \ref{general_prob} on a Riemannian manifold $M$}
\begin{algorithmic}[1]\label{CGM}
\STATE Choose an initial point $x_0\in M$.
\STATE Set $\eta_0=-\grad f(x_0)$.
\FOR{$k=0,1,2,\ldots$}
\STATE Compute the step size $\alpha_k>0$ satisfying the strong Wolfe conditions \eqref{wolfem1} and \eqref{wolfem2} with $0<c_1<c_2<1/2$.
Set
\begin{equation}
x_{k+1}=R_{x_k}\left(\alpha_k\eta_k\right),
\end{equation}
where $R$ is a retraction on $M$.
\STATE Set
\begin{equation}
\beta_{k+1}=\frac{\norm{\grad f(x_{k+1})}_{x_{k+1}}^2}{\norm{\grad f(x_k)}_{x_k}^2},\label{beta0}
\end{equation}
\begin{equation}
\eta_{k+1}=-\grad f(x_{k+1})+\beta_{k+1}\mathcal{T}^R_{\alpha_k\eta_k}(\eta_k),\label{eta0}
\end{equation}
where $\mathcal{T}^R$ is the differentiated retraction defined by \eqref{dr}.
\ENDFOR
\end{algorithmic}
\end{algorithm}

In \cite{ring}, the convergence property of Algorithm \ref{CGM} is verified under the assumption that the inequality
\begin{equation}
\norm{\mathcal{T}^R_{\alpha_k\eta_k}(\eta_k)}_{x_{k+1}}\le\norm{\eta_k}_{x_k}\label{normineq}
\end{equation}
holds for all $k\in\mathbb N$.
However, the assumption does not always hold in general.
For example, the assumption does not hold on the sphere endowed with the orthographic retraction \cite{absil_malick}.
In Section \ref{experiments}, we will numerically treat such a case.

We wish to relax the assumption \eqref{normineq} by using a scaled vector transport.
An idea for improving Algorithm \ref{CGM} is to replace $\mathcal{T}^R$ by the scaled vector transport $\mathcal{T}^0$ defined by \eqref{scaled}.
However, this causes difficulty in computing effectively a step size $\alpha_k$ satisfying \eqref{wolfem4} with $\mathcal{T}=\mathcal{T}^0$.

A simple but effective idea for improving Algorithm \ref{CGM} is that each step size is always computed so as to satisfy the strong Wolfe conditions \eqref{wolfem1} and \eqref{wolfem2},
but the scaled vector transport $\mathcal{T}^0$ is adopted if it is necessary for the purpose of convergence.
More specifically, we use the scaled vector transport $\mathcal{T}^0$ only if the vector transport $\mathcal{T}^R$ increases the norm of the previous search direction vector, that is, we introduce $\mathcal{T}^{(k)}$ defined by
\begin{equation}
\mathcal{T}^{(k)}_{\alpha_k\eta_k}(\eta_k)=\begin{cases}
\mathcal{T}^R_{\alpha_k\eta_k}(\eta_k),\qquad \text{if}\ \ \norm{\mathcal{T}^R_{\alpha_k\eta_k}(\eta_k)}_{x_{k+1}}\le \norm{\eta_k}_{x_k},
\\
\mathcal{T}^{0}_{\alpha_k\eta_k}(\eta_k),\qquad \text{otherwise},
\end{cases}\label{tkdef}
\end{equation}
as a substitute for $\mathcal{T}^R$ in Step 5 of Algorithm \ref{CGM}.
This idea is realized in the following algorithm.

\begin{algorithm}[H]
\caption{A scaled Fletcher-Reeves type conjugate gradient method for Problem \ref{general_prob} on a Riemannian manifold $M$}
\begin{algorithmic}[1]\label{CGMTk}
\STATE Choose an initial point $x_0\in M$.
\STATE Set $\eta_0=-\grad f(x_0)$.
\FOR{$k=0,1,2,\ldots$}
\STATE Compute the step size $\alpha_k>0$ satisfying the strong Wolfe conditions \eqref{wolfem1} and \eqref{wolfem2} with $0<c_1<c_2<1/2$.
Set
\begin{equation}
x_{k+1}=R_{x_k}\left(\alpha_k\eta_k\right),
\end{equation}
where $R$ is a retraction on $M$.
\STATE Set
\begin{equation}
\beta_{k+1}=\frac{\norm{\grad f(x_{k+1})}_{x_{k+1}}^2}{\norm{\grad f(x_k)}_{x_k}^2},\label{beta}
\end{equation}
\begin{equation}
\eta_{k+1}=-\grad f(x_{k+1})+\beta_{k+1}\mathcal{T}^{(k)}_{\alpha_k\eta_k}(\eta_k),\label{tkeq}
\end{equation}
where $\mathcal{T}^{(k)}$ is defined by \eqref{tkdef}, and where $\mathcal{T}^R$ and $\mathcal{T}^0$ are the differentiated retraction and the associated scaled vector transport defined by \eqref{dr} and \eqref{scaled}, respectively.
\ENDFOR
\end{algorithmic}
\end{algorithm}
We will prove in Section \ref{analysis} the global convergence property of the proposed algorithm,
and give in Section \ref{experiments} numerical examples in which the inequality \eqref{normineq} does not hold for all $k\in\mathbb N$ but our Algorithm \ref{CGMTk} indeed has an advantage in generating convergent sequences.

\section{Convergence analysis of the new algorithm}\label{analysis}
In this section, we verify the convergence property of Algorithm \ref{CGMTk}.

\subsection{Zoutendijk's theorem}\label{zousec}
Zoutendijk's theorem about a series associated with search directions on $\mathbb R^n$ is not only valid for the conjugate gradient method but also valid for general descent algorithms \cite{nocedal}.
This theorem can be generalized so as to be applicable to a general descent algorithm (Algorithm \ref{algorithm1}) on a Riemannian manifold $M$.
In the same manner as in $\mathbb R^n$, we define on a Riemannian manifold $M$ the angle $\theta_k$ between the steepest descent direction $-\grad f(x_k)$ and the search direction $\eta_k$ through
\begin{equation}
\cos\theta_k=-\frac{\langle\grad f(x_k),\eta_k\rangle_{x_k}}{\norm{\grad f(x_k)}_{x_k}\norm{\eta_k}_{x_k}}.\label{cosdef}
\end{equation}
Then, Zoutendijk's theorem on $M$ is stated as follows:
\begin{Thm}\label{zoutendijk}
Suppose that in Algorithm \ref{algorithm1} on a Riemannian manifold $M$,
a descent direction $\eta_k$ and a step size $\alpha_k$ satisfy the strong Wolfe conditions \eqref{wolfem1} and \eqref{wolfem2}.
If the objective function $f$ is bounded below and of $C^1$-class, and if there exists a Lipschitzian constant $L>0$ such that 
\begin{equation}
\abs{\D(f\circ R_x)(t\eta)[\eta]-\D(f\circ R_x)(0)[\eta]}\le Lt, \qquad \eta\in T_xM\ {\rm with}\ \norm{\eta}_x=1,\ x\in M,\ t\ge 0,\label{lip}
\end{equation}
then the following series converges;
\begin{equation}
\sum_{k=0}^\infty \cos^2\theta_k\norm{\grad f(x_k)}_{x_k}^2<\infty.\label{zoucos}
\end{equation}
\end{Thm}
The proof of this theorem can be performed in the same manner as that for Zoutendijk's theorem on $\mathbb R^n$.
See \cite{ring} for more detail.

\begin{Remark}
We remark that the inequality \eqref{lip} is a weaker condition than the Lipschitz continuous differentiability of $f\circ R_x$.
We will show in Appendix \ref{app} that Eq.~\eqref{lip} holds for objective functions in practical Riemannian optimization problems.
A further discussion on the relation with the standard Lipschitz continuous differentiability will be also made in the same appendix.
\end{Remark}

\subsection{Global convergence}
We first extend a lemma in \cite{al} so as to be applicable to Algorithm \ref{CGMTk} as follows:

\begin{Lemma}\label{Lemma}
The search direction $\eta_k$ determined in Algorithm \ref{CGMTk} is a descent direction satisfying
\begin{equation}
-\frac{1}{1-c_2}\le \frac{\langle\grad f(x_k),\eta_k\rangle_{x_k}}{\norm{\grad f(x_k)}_{x_k}^2}\le \frac{2c_2-1}{1-c_2}.\label{lemma}
\end{equation}
\end{Lemma}

\begin{proof}
The proof runs by induction.
For $k=0$, the inequality \eqref{lemma} clearly holds on account of
\begin{equation}
\frac{\langle\grad f(x_0),\eta_0\rangle_{x_0}}{\norm{\grad f(x_0)}_{x_0}^2}=\frac{\langle\grad f(x_0),-\grad f(x_0)\rangle_{x_0}}{\norm{\grad f(x_0)}_{x_0}^2}=-1.
\end{equation}
We here note that $0<c_1<c_2<1/2$.
Suppose that $\eta_k$ is a descent direction satisfying \eqref{lemma} for some $k$.
Note that on account of Eq.~\eqref{tkeq} with Eq.~\eqref{tkdef},
$\mathcal{T}^R$ and $\mathcal{T}^{(k)}$ are related by $\norm{\mathcal{T}^{(k)}_{\alpha_k\eta_k}(\eta_k)}_{x_{k+1}}\le \norm{\mathcal{T}^R_{\alpha_k\eta_k}(\eta_k)}_{x_{k+1}}$ in each case.
Since $\mathcal{T}_{\alpha_k\eta_k}^{(k)}(\eta_k)$ and $\mathcal{T}^R_{\alpha_k\eta_k}(\eta_k)$ are in the same direction with the inequality $\norm{\mathcal{T}_{\alpha_k\eta_k}^{(k)}(\eta_k)}_{x_{k+1}}\le \norm{\mathcal{T}^R_{\alpha_k\eta_k}(\eta_k)}_{x_{k+1}}$ in norm, we have
\begin{equation}
\abs{\langle\grad f(x_{k+1}),\mathcal{T}^{(k)}_{\alpha_k\eta_k}(\eta_k)\rangle_{x_{k+1}}}\le\abs{\langle\grad f(x_{k+1}),\mathcal{T}^R_{\alpha_k\eta_k}(\eta_k)\rangle_{x_{k+1}}}.\label{tkandt}
\end{equation}
We also note that the vector transport $\mathcal{T}^R$ is defined to be $\mathcal{T}^R_{\eta_x}(\xi_x)=\D R_x(\eta_x)[\xi_x]$ in the algorithm.
It then follows from \eqref{wolfem2} and \eqref{tkandt} that
\begin{equation}
c_2\langle\grad f(x_k),\eta_k\rangle_{x_k}\le \langle\grad f(x_{k+1}),\mathcal{T}^{(k)}_{\alpha_k\eta_k}(\eta_k)\rangle_{x_{k+1}}\le -c_2\langle\grad f(x_k),\eta_k\rangle_{x_k},\label{lemma1}
\end{equation}
where it is to be noted that $\eta_k$ is in a descent direction.
The middle term in \eqref{lemma} with $k+1$ for $k$ is computed as
\begin{align}
\frac{\langle\grad f(x_{k+1}),\eta_{k+1}\rangle_{x_{k+1}}}{\norm{\grad f(x_{k+1})}_{x_{k+1}}^2}
=&\frac{\langle\grad f(x_{k+1}),-\grad f(x_{k+1})+\beta_{k+1}\mathcal{T}^{(k)}_{\alpha_k\eta_k}(\eta_k)\rangle_{x_{k+1}}}{\norm{\grad f(x_{k+1})}_{x_{k+1}}^2}\notag\\
=&-1+\frac{\langle\grad f(x_{k+1}),\mathcal{T}^{(k)}_{\alpha_k\eta_k}(\eta_k)\rangle_{x_{k+1}}}{\norm{\grad f(x_k)}_{x_k}^2},\label{lemma2}
\end{align}
where the definition \eqref{beta} of $\beta_{k+1}$ has been used.
Therefore, we obtain from \eqref{lemma1} and \eqref{lemma2}
\begin{equation}
-1+c_2\frac{\langle\grad f(x_k),\eta_k\rangle_{x_k}}{\norm{\grad f(x_k)}_{x_k}^2}\le \frac{\langle\grad f(x_{k+1}),\eta_{k+1}\rangle_{x_{k+1}}}{\norm{\grad f(x_{k+1})}_{x_{k+1}}^2}\le -1-c_2\frac{\langle\grad f(x_k),\eta_k\rangle_{x_k}}{\norm{\grad f(x_k)}_{x_k}^2}.
\end{equation}
The inequality \eqref{lemma} for $k+1$ immediately follows from the induction hypothesis.
\end{proof}

We proceed to the global convergence property of Algorithm \ref{CGMTk}.
The convergence of the conjugate gradient method has been already proved on $\mathbb R^n$ by Al-Baali \cite{al}.
Exploiting the idea of the proof used in \cite{al}, we show that Algorithm \ref{CGMTk} generates converging sequences on a Riemannian manifold.

\begin{Thm}\label{thm_glocon}
Consider Algorithm \ref{CGMTk}.
If \eqref{lip} and hence \eqref{zoucos} hold, then
\begin{equation}
\liminf_{k\to\infty}\norm{\grad f(x_k)}_{x_k}=0.\label{glocon}
\end{equation}
\end{Thm}

\begin{proof}
If $\grad f(x_k)=0$ for some $k$, let $k_0$ be the smallest integer among such $k$.
Then, we have $\beta_{k_0}=0$ and $\eta_{k_0}=0$ from \eqref{beta} and \eqref{tkeq} with $k_0=k+1$, so that $x_{k_0+1}=R_{x_{k_0}}(\alpha_{k_0}\eta_{k_0})=R_{x_{k_0}}(0)=x_{k_0}$.
It then follows that $\grad f(x_k)=0$ for all $k\ge k_0$.
Eq.~\eqref{glocon} clearly holds in such a case.

We shall consider the case in which $\grad f(x_k)\neq 0$ for all $k$ and prove \eqref{glocon} by contradiction.
Assume that \eqref{glocon} does not hold, that is, there exists a constant $\gamma>0$ such that
\begin{equation}
\norm{\grad f(x_k)}_{x_k}\ge \gamma>0,\qquad \forall k\ge 0.\label{bound}
\end{equation}
Now from \eqref{cosdef} and \eqref{lemma}, we obtain
\begin{equation}
\cos\theta_k\ge \frac{1-2c_2}{1-c_2}\frac{\norm{\grad f(x_k)}_{x_k}}{\norm{\eta_k}_{x_k}}.\label{cosineq}
\end{equation}
On account of Thm.~\ref{zoutendijk}, Eqs.~\eqref{zoucos} and \eqref{cosineq} are put together to provide
\begin{equation}
\sum_{k=0}^{\infty}\frac{\norm{\grad f(x_k)}_{x_k}^4}{\norm{\eta_k}_{x_k}^2}<\infty.\label{inf}
\end{equation}
On the other hand, Eqs.~\eqref{tkandt}, \eqref{lemma}, and the strong Wolfe condition \eqref{wolfem2} are put together to give
\begin{align}
\abs{\langle\grad f(x_k),\mathcal{T}^{(k-1)}_{\alpha_{k-1}\eta_{k-1}}(\eta_{k-1})\rangle_{x_k}}\le & \abs{\langle\grad f(x_k),\mathcal{T}^R_{\alpha_{k-1}\eta_{k-1}}(\eta_{k-1})\rangle_{x_k}}\notag\\
\le &-c_2\langle \grad f(x_{k-1}),\eta_{k-1}\rangle_{x_{k-1}}\notag\\
\le & \frac{c_2}{1-c_2}\norm{\grad f(x_{k-1})}_{x_{k-1}}^2.
\end{align}
Using this inequality and the definition of $\beta_k$, we obtain the recurrence inequality for $\norm{\eta_k}^2_{x_k}$ as follows:
\begin{align}
&\norm{\eta_k}_{x_k}^2\notag\\
=&\norm{-\grad f(x_k)+\beta_k\mathcal{T}^{(k-1)}_{\alpha_{k-1}\eta_{k-1}}(\eta_{k-1})}_{x_k}^2\notag\\
\le &\norm{\grad f(x_k)}_{x_k}^2+2\beta_k\abs{\langle\grad f(x_k),\mathcal{T}^{(k-1)}_{\alpha_{k-1}\eta_{k-1}}(\eta_{k-1})\rangle_{x_k}}+\beta_k^2\norm{\mathcal{T}^{(k-1)}_{\alpha_{k-1}\eta_{k-1}}(\eta_{k-1})}_{x_k}^2\notag\\
\le &\norm{\grad f(x_k)}_{x_k}^2+\frac{2c_2}{1-c_2}\beta_k\norm{\grad f(x_{k-1})}_{x_{k-1}}^2+\beta_k^2\norm{\eta_{k-1}}_{x_{k-1}}^2\notag\\
= &c\norm{\grad f(x_k)}_{x_k}^2+\beta_k^2\norm{\eta_{k-1}}_{x_{k-1}}^2,\label{dots}
\end{align}
where we have used the fact that $\norm{\mathcal{T}^{(k-1)}_{\alpha_{k-1}\eta_{k-1}}(\eta_{k-1})}_{x_k}\le\norm{\eta_{k-1}}_{x_{k-1}}$ and put \\
$c:=(1+c_2)/(1-c_2)>1$.
The successive use of this inequality together with the definition of $\beta_k$ results in
\begin{align}
&\norm{\eta_k}_{x_k}^2\notag\\
\le &c\left(\norm{\grad f(x_k)}_{x_k}^2+\beta_k^2\norm{\grad f(x_{k-1})}_{x_{k-1}}^2+\cdots+\beta_k^2\beta_{k-1}^2\cdots\beta_2^2\norm{\grad f(x_1)}_{x_1}^2\right)\notag\\
&+\beta_k^2\beta_{k-1}^2\cdots\beta_1^2\norm{\eta_0}_{x_0}^2\notag\\
=&c\norm{\grad f(x_k)}_{x_k}^4\left(\norm{\grad f(x_k)}_{x_k}^{-2}+\norm{\grad f(x_{k-1})}_{x_{k-1}}^{-2}+\cdots+\norm{\grad f(x_1)}_{x_1}^{-2}\right)\notag\\
&+\norm{\grad f(x_k)}_{x_k}^4\norm{\grad f(x_0)}_{x_0}^{-2}\notag\\
<&c\norm{\grad f(x_k)}_{x_k}^4\sum_{j=0}^k\norm{\grad f(x_j)}_{x_j}^{-2}\le \frac{c}{\gamma^2}\norm{\grad f(x_k)}_{x_k}^4(k+1),\label{asta}
\end{align}
where use has been made of \eqref{bound} in the last inequality.
The inequality \eqref{asta} gives rise to
\begin{equation}
\sum_{k=0}^{\infty}\frac{\norm{\grad f(x_k)}_{x_k}^4}{\norm{\eta_k}_{x_k}^2}\ge \sum_{k=0}^{\infty}\frac{\gamma^2}{c}\frac{1}{k+1}=\infty.
\end{equation}
This contradicts \eqref{inf} and the proof is completed.
\end{proof}

\section{Numerical experiments}\label{experiments}
In this section, we compare Algorithm \ref{CGMTk} with Algorithm \ref{CGM} by numerical experiments.
As is shown in \cite{ring}, if the vector transport $\mathcal{T}^R$ as the differentiated retraction satisfies the inequality \eqref{normineq},
the convergence property of Algorithm \ref{CGM} is proved.
However, if \eqref{normineq} does not hold, it is not always ensured that sequences generated by Algorithm \ref{CGM} converge.
In contrast with this, Algorithm \ref{CGMTk} indeed works well even if \eqref{normineq} fails to hold, as is verified in Thm.~\ref{thm_glocon}.
In the following, we give two examples which show that Algorithm \ref{CGMTk} works better than Algorithm \ref{CGM}.
One of the examples is somewhat artificial but well illustrates the situation in which a sequence generated by Algorithm \ref{CGM} is unlikely to converge.
The other is a more natural example encountered in a practical problem.

In both of two examples, we consider the following Rayleigh quotient minimization problem on the sphere $S^{n-1}:=\left\{x\in\mathbb R^n\,|\,x^Tx=1\right\}$ \cite{absil, helmke}:
\begin{Prob}\label{rayleigh_prob}
\begin{align}
{\rm minimize} \,\,\,\,\,\,& f(x)=x^TAx,\label{func_sphere}\\
{\rm subject\,\,to} \,\,\,& x\in S^{n-1},
\end{align}
\end{Prob}
\noindent where $A:=\diag(\lambda_1,\lambda_2,\ldots,\lambda_n)$ with $\lambda_1<\lambda_2<\cdots<\lambda_n$.
The optimal solutions of this problem are $\pm(1,0,0,\ldots,0)^T$, which are the unit eigenvectors of $A$ associated with the smallest eigenvalue $\lambda_1$.

\subsection{A sphere endowed with a peculiar metric}\label{sphere1}
Consider Problem \ref{rayleigh_prob} with $n=20$ and $A=\diag(1,2,\ldots,20)$.
A Riemannian metric $g(\cdot,\cdot)$ on $S^{n-1}$ is here defined by
\begin{equation}
g_x(\xi_x,\eta_x):=\xi_x^TG_x\eta_x,\qquad \xi_x,\eta_x\in T_xS^{n-1},\label{met_sphere}
\end{equation}
where $G_x:=\diag(10000(x^{(1)})^2+1,1,1,\ldots,1)$, and where $x^{(1)}$ denotes the first component of the column vector $x$.
It is to be noted that this metric is not the standard one on $S^{n-1}$.
The norm $\norm{\xi_x}_x$ of $\xi_x\in T_x S^{n-1}$ is then defined to be $\norm{\xi_x}_x=\sqrt{g_x(\xi_x,\xi_x)}$.
If $x$ is close to the optimal solutions $\pm(1,0,0,\ldots,0)$, then $(x^{(1)})^2$ is nearly $1$.
Since the first diagonal element of $G_x$ is large because of the coefficient $10000$, the closer $x$ is to $\pm(1,0,0,\ldots,0)$, the larger the norm $\norm{\xi_x}_x$ tends to be.

With respect to the metric \eqref{met_sphere}, the gradient of $f$ is described as
\begin{equation}
\grad f(x)=2\left(I-\frac{G_x^{-1}xx^T}{x^TG_x^{-1}x}\right)G_x^{-1}Ax.\label{experiment_grad}
\end{equation}
Indeed, the right-hand side of \eqref{experiment_grad} belongs to $T_x S^{n-1}=\left\{\xi\in\mathbb R^{n}\,|\,x^T\xi=0\right\}$ and it holds that
\begin{equation}
g_x\left( 2\left(I-\frac{G_x^{-1}xx^T}{x^TG_x^{-1}x}\right)G_x^{-1}Ax,\ \xi\right)=2x^TA\xi=\D f(x)[\xi]
\end{equation}
for any $\xi\in T_x S^{n-1}$.
Let $R$ be the retraction on $S^{n-1}$ defined by
\begin{equation}
R_x(\xi)=\frac{x+\xi}{\sqrt{(x+\xi)^T(x+\xi)}},\qquad \xi\in T_x S^{n-1},\ x\in S^{n-1},
\end{equation}
which is the special case of the QR retraction \eqref{QRret} on the Stiefel manifold defined in Appendix \ref{app}.
For this $R$, the differentiated retraction $\mathcal{T}^R$ defined by \eqref{dr} is written out as
\begin{equation}
\mathcal{T}^R_{\eta}(\xi)=\frac{1}{\sqrt{(x+\eta)^T(x+\eta)}}\left(I-\frac{(x+\eta)(x+\eta)^T}{(x+\eta)^T(x+\eta)}\right)\xi,\qquad \eta,\xi\in T_x S^{n-1},\ x\in S^{n-1}.
\end{equation}

We note that though the metric endowed with is not the standard one, the Lipschitzian condition \eqref{lip} holds, as is mentioned in Rem.~\ref{remark_app2} in Appendix \ref{app}.
Hence from Thm.~\ref{thm_glocon}, Algorithm \ref{CGMTk} works well in theory.

\begin{figure}[htbp]
  \begin{center}
   \includegraphics[width=100mm]{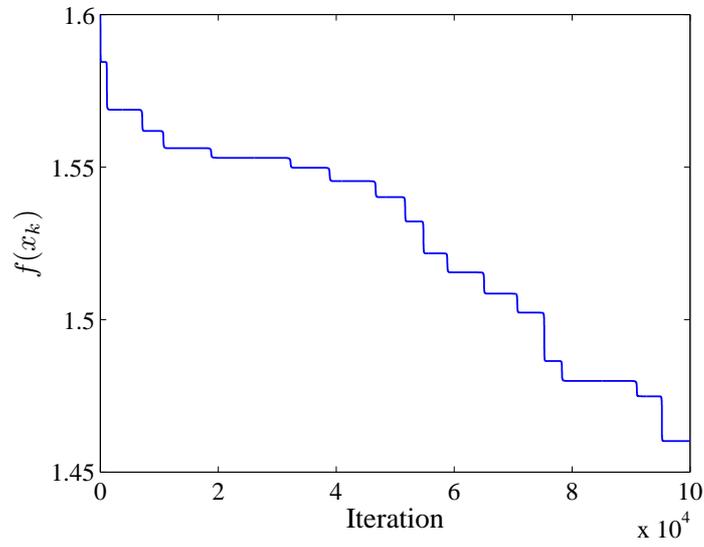}
  \end{center}
  \caption{The sequence of the values $f(x_k)$ of the objective function $f$ evaluated on the sequence $\left\{x_k\right\}$ generated by Algorithm \ref{CGM}.}
  \label{fig:f}
\end{figure}
\begin{figure}[htbp]
  \begin{center}
   \includegraphics[width=100mm]{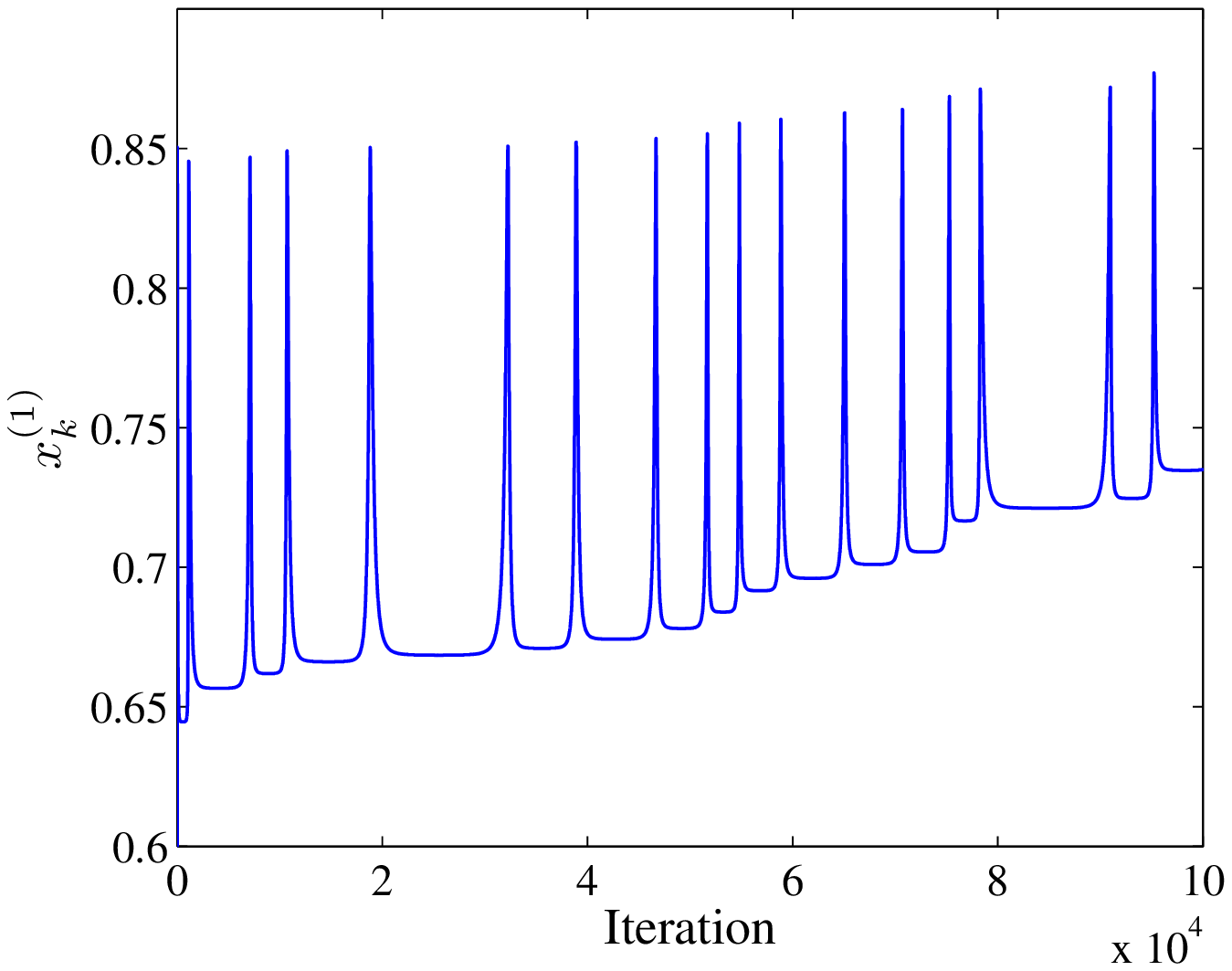}
  \end{center}
  \caption{The sequence of the first components $x_k^{(1)}$ from the sequence $\left\{x_k\right\}$ generated by Algorithm \ref{CGM}.}
  \label{fig:1}
\end{figure}
\begin{figure}[htbp]
  \begin{center}
   \includegraphics[width=100mm]{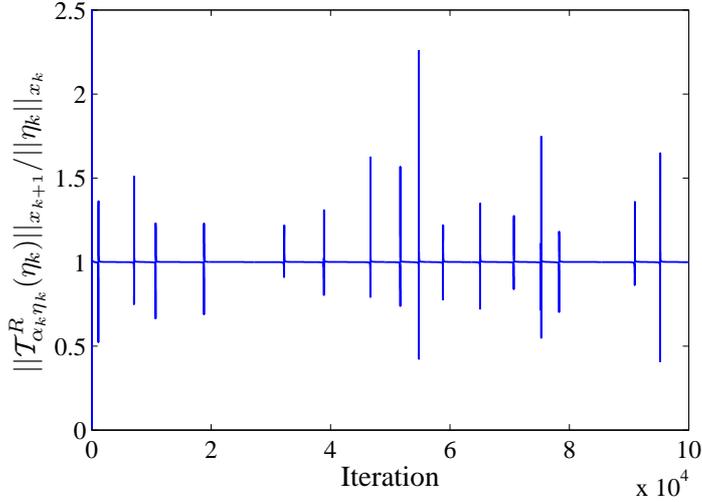}
  \end{center}
  \caption{Ratios $\norm{\mathcal{T}^R_{\alpha_k\eta_k}(\eta_k)}_{x_{k+1}}/\norm{\eta_k}_{x_k}$ evaluated on the sequences $\left\{x_k\right\}$ and $\left\{\eta_k\right\}$ generated by Algorithm \ref{CGM}.}
  \label{fig:2}
\end{figure}

Figs.~\ref{fig:f}, \ref{fig:1}, and \ref{fig:2} show numerical results from applying Algorithm \ref{CGM} to Problem \ref{rayleigh_prob} with the initial point $x_0=(1,1,\ldots,1)^T/2\sqrt{5}\in S^{n-1}$ with $n=20$.
The vertical axes of Figs.~\ref{fig:f}, \ref{fig:1}, and \ref{fig:2} carry values of $f(x_k)$ at $x_k$, values of the first components $x_k^{(1)}$ of $x_k$, and values of the ratios $\norm{\mathcal{T}^R_{\alpha_k\eta_k}(\eta_k)}_{x_{k+1}}/\norm{\eta_k}_{x_k}$, respectively.
Note that for the optimal solution $x_*=(1,0,0,\ldots,0)^T\in S^{n-1}$ which the current generated sequence $\left\{x_k\right\}$ is expected to approach, the target value is $f(x_*)=x_*^{(1)}=1$ in both Figs.~\ref{fig:f} and \ref{fig:1}.
Though the $\left\{x_k\right\}$ seems to come close to $x_*$ bit by bit, the convergence is not observed even after $10^5$ iterations.
At the iteration number $10^5$, $f(x_k)$ is far from $f(x_*)=1$, as is seen from Fig.~\ref{fig:f}.
Fig.~\ref{fig:1} shows that the sequence is intermittently repelled from the target point, when approaching it.
If more iterations, say $10^7$, are performed, the graph of $\{x_k^{(1)}\}$ has almost the same shape, that is, sharp peaks repeatedly appear in Fig.~\ref{fig:1} with extended iterations.
If $\norm{\mathcal{T}^R_{\alpha_k\eta_k}(\eta_k)}_{x_{k+1}}/\norm{\eta_k}_{x_k}\le 1$ for all $k\in\mathbb N$, the sequence $\left\{x_k\right\}$ would converge.
However, as is shown in Fig.~\ref{fig:2}, the ratio $\norm{\mathcal{T}^R_{\alpha_k\eta_k}(\eta_k)}_{x_{k+1}}/\norm{\eta_k}_{x_k}$ intermittently exceeds the value $1$.
This fact seems to prevent the sequence from converging, as long as numerical experiments suggest.
To gain insight into the non-convergence problem, we put Figs.~\ref{fig:1} and \ref{fig:2} together into Fig.~\ref{fig:3}, which shows that the peaks of two graphs synchronize.
\begin{figure}[htbp]
  \begin{center}
   \includegraphics[width=100mm]{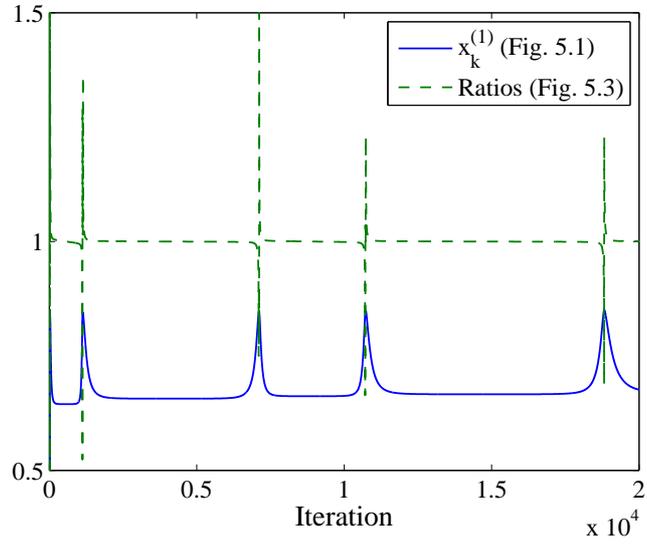}
  \end{center}
  \caption{$x_k^{(1)}$ and $\norm{\mathcal{T}^R_{\alpha_k\eta_k}(\eta_k)}_{x_{k+1}}/\norm{\eta_k}_{x_k}$ by Algorithm \ref{CGM}.}
  \label{fig:3}
\end{figure}
This suggests that the violation of the inequality \eqref{normineq} makes the sequence fail to approach the optimal solution $x_*$.
This phenomenon is caused by the large first diagonal element of $G_x$ in the neighbourhood of $x_*$.

\begin{figure}[htbp]
  \begin{center}
   \includegraphics[width=100mm]{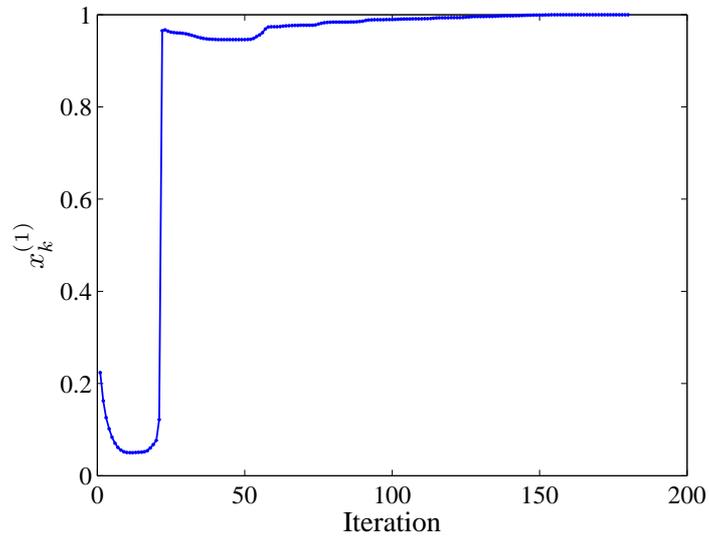}
  \end{center}
  \caption{The sequence of the first components $x_k^{(1)}$ from the sequence $\left\{x_k\right\}$ generated by Algorithm \ref{CGMTk}.}
  \label{fig:4}
\end{figure}
In contrast with this, in Algorithm \ref{CGMTk}, the vector transport $\mathcal{T}^R$ is scaled if necessary, and thereby generated sequences converge to solve Problem \ref{rayleigh_prob}.
In comparison with Fig.~\ref{fig:1}, Fig.~\ref{fig:4} shows that the present algorithm generates a converging sequence, resolving the difficulty of being repelled from the optimal solution.
We here note that the inequality $\norm{\mathcal{T}_{\alpha_k\eta_k}^{(k)}(\eta_k)}_{x_{k+1}}\le\norm{\eta_k}_{x_k}$ is never violated in this algorithm.

We now investigate the performance of Algorithm \ref{CGMTk} in more detail with interest in comparison with a restart strategy in the conjugate gradient method.
\begin{figure}[htbp]
  \begin{center}
   \includegraphics[width=100mm]{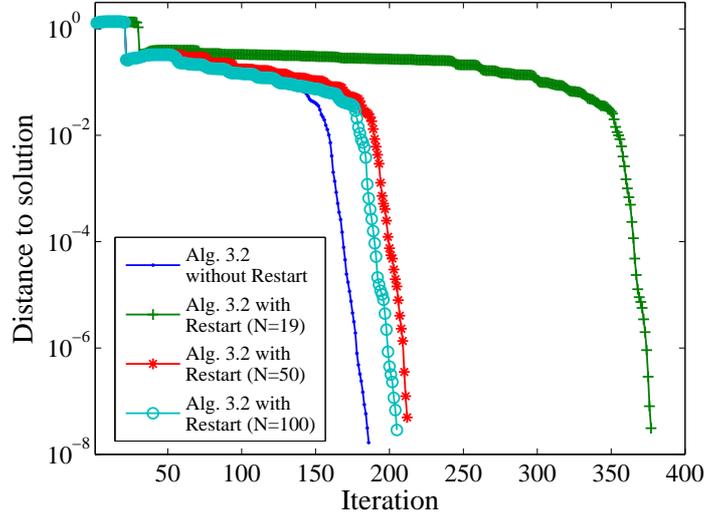}
  \end{center}
  \caption{The sequences of the distances between $x_k$ and $x_*$ with respect to the sequences $\left\{x_k\right\}$ generated by Algorithm \ref{CGMTk} with several restarting strategies.}
  \label{fig:6}
\end{figure}
As is well known, in a nonlinear conjugate gradient method on the Euclidean space,
the iteration is often restarted at every $N$ steps by taking a steepest descent search direction,
where $N$ is usually chosen to be the dimension of the search space in the problem.
To gain a sight of the performance of the restart method on a Riemannian manifold, we introduce a similar restart strategy into Algorithms \ref{CGM} and \ref{CGMTk}, that is,
we set $\beta_{k+1}=0$ in Step 5 of each algorithm at every $N$ steps.
A choice for $N$ is $19$, which is the dimension of $S^{n-1}$ with $n=20$.
For comparison, the both algorithms with restarts are also performed for $N=50$ and $N=100$.
The results from Algorithm \ref{CGMTk} with and without restart are shown in Fig.~\ref{fig:6}.
The vertical axis of Fig.~\ref{fig:6} carries $\sqrt{(x_k-x_*)^T(x_k-x_*)}$, which is an approximation of the distance between $x_k$ and $x_*$ on $S^{n-1}$.
We can observe from the graphs in Fig.~\ref{fig:6} that Algorithm \ref{CGMTk} with and without restart has a superlinear convergence property.
Fig.~\ref{fig:6} shows further that Algorithm \ref{CGMTk} without restart exhibits better performance than Algorithm \ref{CGMTk} with a few variants of restarts,
which means that the restart strategy fails to improve the performance of Algorithm \ref{CGMTk}.

\begin{figure}[htbp]
  \begin{center}
   \includegraphics[width=100mm]{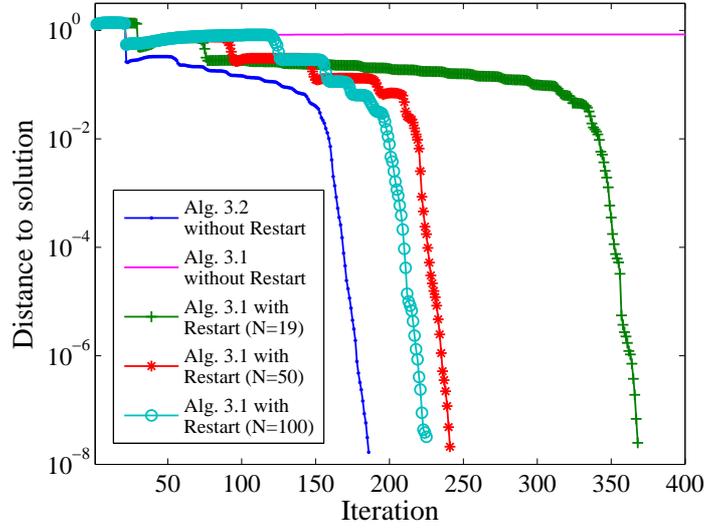}
  \end{center}
  \caption{The sequences of the distances between $x_k$ and $x_*$ with respect to the sequences $\left\{x_k\right\}$ generated by Algorithm \ref{CGMTk} and Algorithm \ref{CGM} with several restarting strategies.}
  \label{fig:7}
\end{figure}
On the contrary, the restart strategy improves the performance of Algorithm \ref{CGM},
but the resultant performance is not comparable to Algorithm \ref{CGMTk} without restart yet.
A numerical evidence is shown in Fig.~\ref{fig:7}.

\subsection{The sphere endowed with the orthographic retraction}
We give a more natural example, in which the inequality \eqref{normineq} is never satisfied.
Consider Problem \ref{rayleigh_prob} with $n=100$ and $A=\diag(1,2,\ldots,100)/100$.
The difference from the example in Subsection \ref{sphere1} is the choice of a Riemannian metric and a retraction.
We in turn endow the sphere $S^{n-1}$ with the induced metric $\langle\cdot,\cdot\rangle$ from the natural inner product on $\mathbb R^n$:
\begin{equation}
\langle\xi_x,\eta_x\rangle_x:=\xi_x^T\eta_x,\qquad \xi_x,\eta_x\in T_xS^{n-1}.\label{nat_met_sphere}
\end{equation}
The norm of $\xi_x\in T_x S^{n-1}$ is then defined to be $\norm{\xi_x}_x=\sqrt{\xi_x^T\xi_x}$ as usual.
With the natural metric $\langle\cdot,\cdot\rangle$, the gradient of $f$ is written out as
\begin{equation}
\grad f(x)=2(I-xx^T)Ax.
\end{equation}
We consider the orthographic retraction $R$ on $S^{n-1}$ \cite{absil_malick}, which is defined to be
\begin{equation}
R_x(\xi)=\sqrt{1-\xi^T\xi}\,x+\xi,\qquad \xi\in T_x S^{n-1}\ \text{with}\ \norm{\xi}_x< 1.
\end{equation}
Associated with this $R$, the vector transport $\mathcal{T}^R$ is written out as
\begin{equation}
\mathcal{T}^R_{\eta}(\xi)=\xi-\frac{\eta^T\xi}{\sqrt{1-\eta^T\eta}}x,\qquad \eta,\xi\in T_x S^{n-1}\ \text{with}\ \norm{\eta}_x,\norm{\xi}_x< 1,\ x\in S^{n-1}.
\end{equation}
For this $\mathcal{T}^R$, the norm $\norm{\mathcal{T}^R_{\eta}(\xi)}_{R_x(\eta)}$ is evaluated as
\begin{equation}
\norm{\mathcal{T}^R_{\eta}(\xi)}_{R_x(\eta)}^2=\norm{\xi}_x^2+\frac{(\eta^T\xi)^2}{1-\norm{\eta}_x^2}\ge\norm{\xi}_x^2,
\end{equation}
where use has been made of $x^Tx=1$ and $x^T\xi=0$.
Thus, the inequality \eqref{normineq}, which is the key condition for the proof of the global convergence property of Algorithm \ref{CGM}, is violated unless $\eta_k=0$.
In spite of this fact, we may try to perform Algorithm \ref{CGM} for this problem.
If the generated sequence does not diverge, we can compare the result with that obtained by Algorithm \ref{CGMTk}.
We performed Algorithms \ref{CGM} and \ref{CGMTk} and obtained Fig.~\ref{fig:8}, whose vertical axis carries $\sqrt{(x_k-x_*)^T(x_k-x_*)}$.
The figure shows the superiority of the proposed algorithm.
\begin{figure}[htbp]
  \begin{center}
   \includegraphics[width=100mm]{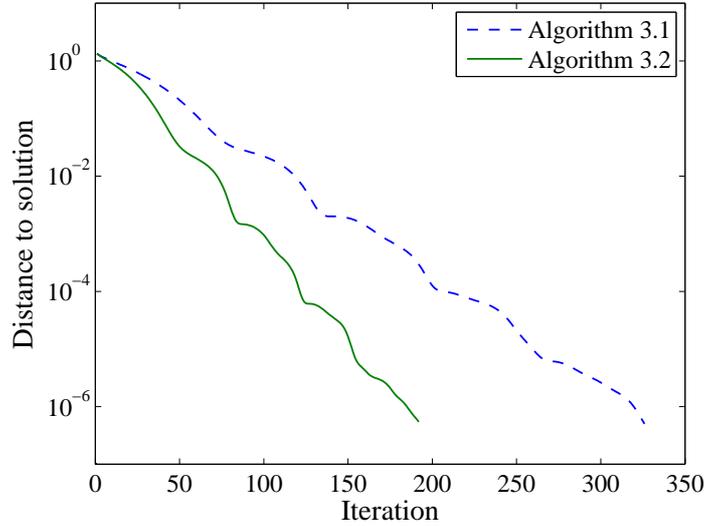}
  \end{center}
  \caption{The sequences of distances between $x_k$ and $x_*$ for the sequences $\left\{x_k\right\}$ generated by Algorithms \ref{CGM} and \ref{CGMTk} with the orthographic retraction.}
  \label{fig:8}
\end{figure}

\section{Concluding Remarks}\label{conclusion}
We have dealt with the global convergence of the conjugate gradient method with the Fletcher-Reeves $\beta$.
Though the conjugate gradient method generates globally converging sequences in the Euclidean space, the conjugate gradient method on a Riemannian manifold $M$ has not been shown to have a convergence property in general, but under the assumption that the vector transport $\mathcal{T}^R$ as the differentiated retraction does not increase the norm of the tangent vector, the convergence is proved in \cite{ring}.
If the parallel translation is adopted as a vector transport, the conjugate gradient method is shown to generate converging sequences, as is given in \cite{smith}.
However, the parallel translation is not convenient for computational effectiveness.
For computational efficiency, we have introduced a vector transport, in place of the parallel translation, with a modification that the vector transport $\mathcal{T}^R$ is replaced by the scaled vector transport $\mathcal{T}^0$ only when $\mathcal{T}^R$ increases the norm of the search direction vector.
The idea is simple but effective.
We have achieved a balance between computational efficiency and the global convergence by proposing Algorithm \ref{CGMTk}.
We have shown the convergence of the present algorithm both in the theoretical and the numerical viewpoints.
In particular, we have performed numerical experiments to show that the present algorithm can solve problems for which the existing algorithm cannot work well because of the violation of the assumption about the vector transport.

\begin{appendix}
\section{Examples in which the condition \eqref{lip} holds}\label{app}
In Thm.~\ref{zoutendijk}, we assume that the condition \eqref{lip} holds.
We here compare \eqref{lip} with the condition that $f\circ R_x$ is Lipschitz continuously differentiable uniformly for $x$, that is, there exists a Lipschitz constant $L>0$ such that
\begin{equation}
\norm{\D (f\circ R_x)(\xi)-\D (f\circ R_x)(\zeta)}\le L\norm{\xi-\zeta}_x,\qquad \xi,\zeta\in T_x M, x\in M,\label{appendixlip}
\end{equation}
where the $\norm{\cdot}$ of the left-hand side means the operator norm (see \cite{ring} for detail).
The condition \eqref{appendixlip} is equivalent to
\begin{equation}
\sup_{\norm{\eta}_x=1} \abs{(\D (f\circ R_x)(\xi)-\D (f\circ R_x)(\zeta))[\eta]}\le L\norm{\xi-\zeta}_x,\qquad \xi,\zeta\in T_x M, x\in M.\label{appendixlip2}
\end{equation}
In particular, setting $\zeta=0$ and $\xi=t\eta$ in \eqref{appendixlip2} yields \eqref{lip}.
In this sense, the condition \eqref{lip} is a weaker form of \eqref{appendixlip}.
The assumption \eqref{lip} is of practical use.
For example, the problem of minimizing the Brockett cost function on the Stiefel manifold $\St(p,n)$ with the natural induced metric \cite{absil} has this property, as is shown below.

Let $n, p$ be positive integers with $n\ge p$.
The Stiefel manifold $\St(p,n)$ is defined to be $\St(p,n):=\left\{X\in\mathbb R^{n\times p}\,|\,X^TX=I_p\right\}$.
We consider $\St(p,n)$ as a Riemannian submanifold of $\mathbb R^{n\times p}$ endowed with the natural induced metric
\begin{equation}
\langle\xi,\eta\rangle_X:=\tr(\xi^T\eta),\qquad \xi,\eta\in T_X\!\St(p,n).\label{stmet}
\end{equation}
Let $A$ be an $n\times n$ symmetric matrix and $N:=\diag(\mu_1,\mu_2,\ldots,\mu_p)$ with $0<\mu_1<\mu_2<\cdots<\mu_p$.
The Brockett cost function $f$ is defined on $\St(p,n)$ to be
\begin{equation}
f(X)=\tr\left(X^TAXN\right).\label{bro}
\end{equation}
Further, the QR decomposition-based retraction (which we call the QR retraction) $R$ is defined to be
\begin{equation}
R_X(\xi):=\qf(X+\xi),\qquad \xi\in T_X\!\St(p,n),\ X\in \St(p,n),\label{QRret}
\end{equation}
where $\qf(B)$ denotes the Q-factor of the QR decomposition of a full rank matrix $B\in\mathbb R^{n\times p}$.
That is, if $B$ is decomposed into $B=QR$, where $Q\in\St(p,n)$ and $R$ is an upper triangular $p\times p$ matrix with positive diagonal elements, then $\qf(B)=Q$.

\begin{Prop}\label{prop_app1}
The inequality \eqref{lip} holds for the Brockett cost function \eqref{bro} on $M=\St(p,n)$, where $\St(p,n)$ is endowed with the natural induced metric \eqref{stmet}, and where the QR retraction \eqref{QRret} is adopted.
\end{Prop}

\begin{proof}
Since the function \eqref{bro} is smooth, we have only to show that
\begin{equation}
\left|\frac{d^2}{dt^2}\left(f\circ R_X\right)(t\eta)\right|\le L,\qquad \eta\in T_X\!\St(p,n)\ \text{with}\ \norm{\eta}_X=1,\ X\in \St(p,n),\ t\ge 0.\label{appeq}
\end{equation}
In fact, Eq.~\eqref{lip} is a straightforward consequence of this inequality.
Let $Q(t)$ be a curve defined by $R_X(t\eta)=\qf(X+t\eta)$, and $x_k,\eta_k,q_k(t)$ denote the $k$-th column vectors of $X, \eta, Q(t)$, respectively.
Then, through the Gram-Schmidt orthonormalization process, we obtain
\begin{equation}
q_k(t)=\frac{x_k+t\eta_k-\sum_{i=1}^{k-1}(q_i(t),x_k+t\eta_k)q_i(t)}{\norm{x_k+t\eta_k-\sum_{i=1}^{k-1}(q_i(t),x_k+t\eta_k)q_i(t)}},\label{qk}
\end{equation}
where $(a,b):=a^Tb$ and $\norm{a}:=\sqrt{(a,a)}$ for $n$-dimensional vectors $a,b$.
By induction on $k$, we can take vector-valued polynomials $g_k(t)$ in $t$ satisfying
\begin{equation}
q_k(t)=\frac{g_k(t)}{\norm{g_k(t)}}, \qquad t\ge 0.\label{qkgk}
\end{equation}
Indeed, for $k=1$, \eqref{qkgk} holds with $g_1(t)=x_1+t\eta_1$.
Suppose that \eqref{qkgk} holds for $1,\ldots,k-1$.
Then we can write out $q_k(t)$ as
\begin{equation}
q_k(t)=\frac{\prod_{j=1}^{k-1}\norm{g_j(t)}^2(x_k+t\eta_k)-\sum_{i=1}^{k-1}\prod_{j\neq i}\norm{g_j(t)}^2(g_i(t),x_k+t\eta_k)g_i(t)}{\norm{\prod_{j=1}^{k-1}\norm{g_j(t)}^2(x_k+t\eta_k)-\sum_{i=1}^{k-1}\prod_{j\neq i}\norm{g_j(t)}^2(g_i(t),x_k+t\eta_k)g_i(t)}}.\label{induction_app}
\end{equation}
Denoting by $g_k(t)$ the numerator of the right-hand side of \eqref{induction_app}, which is a polynomial in $t$, we obtain \eqref{qkgk}.

Let
\begin{equation}
h(X,\eta,t)=\frac{d^2}{dt^2}(f\circ R_X)(t\eta).
\end{equation}
Then, the $h(X,\eta,t)$ is written out as
\begin{equation}
h(X,\eta,t)=\sum_{k=1}^p \mu_k\frac{d^2}{dt^2}\left(q_k(t)^TAq_k(t)\right).\label{app_h}
\end{equation}
Since $q_k(t)^TAq_k(t)=g_k(t)^TAg_k(t)/\norm{g_k(t)}^2$, and since the degree of the numerator polynomial in $t$ is not more than that of the denominator polynomial, the degree of the numerator polynomial from the right-hand side of \eqref{app_h} is less than that of the denominator polynomial, so that one has, as $t\to\infty$,
\begin{equation}
\lim_{t\to\infty}h(X,\eta,t)=0.
\end{equation}
This implies that $h(X,\eta,t)$ is bounded with respect to $t\ge 0$.
Moreover, the $h(X,\eta,t)$ is continuous with respect to $X$ and $\eta$ on the compact set $\left\{(X,\eta)\in T\St(p,n)\,|\,\norm{\eta}_X=1\right\}$. 
It then turns out that $h(X,\eta,t)$ is bounded on the whole domain, which implies that there exists $L>0$ such that \eqref{appeq} holds.
This completes the proof.
\end{proof}

\begin{Remark}\label{remark_app1}
Reviewing the proof, we observe that since the QR retraction is irrespective of the metric with which the $\St(p,n)$ is endowed, and since the set $\left\{(X,\eta)\in T\St(p,n)\,|\,\norm{\eta}_X=1\right\}$ is compact with respect to any metric on $\St(p,n)$, the inequality \eqref{lip} with $R$ being the QR retraction \eqref{QRret} holds for the Brockett cost function \eqref{bro} independently of the choice of a metric.
\end{Remark}

\begin{Remark}\label{remark_app2}
We also note that Prop.~\ref{prop_app1} and Rem.~\ref{remark_app1} cover both the Rayleigh quotient on the sphere $S^{n-1}$ as $p=1$ and the Brockett cost function on the orthogonal group as $p=n$.
In particular, the inequality \eqref{lip} holds for the function \eqref{func_sphere}, though the sphere $S^{n-1}$ is endowed with the non-standard metric \eqref{met_sphere}.
\end{Remark}

Another example for \eqref{lip} comes from the problem of minimizing the function
\begin{equation}
F(U,V)=\tr(U^TAVN)\label{satoobj}
\end{equation} on $\St(p,m)\times \St(p,n)$, where $A$ is an $m\times n$ matrix and $N=\diag(\mu_1,\ldots,\mu_p)$ with $\mu_1>\cdots>\mu_p>0$.
An optimal solution to this problem gives the singular value decomposition of $A$ \cite{sato}.
Let $m, n, p$ be positive integers with $m\ge n\ge p$.
We consider $\St(p,m)\times \St(p,n)$ as a Riemannian submanifold of $\mathbb R^{m\times p}\times \mathbb R^{n\times p}$ endowed with the natural induced metric;
\begin{align}
\langle(\xi_1,\eta_1),(\xi_2,\eta_2)\rangle_{(U,V)}:=\tr(\xi_1^T\xi_2)+\tr(\eta_1^T\eta_2),\notag\\
(\xi_1,\eta_1),(\xi_2,\eta_2)\in T_{(U,V)}\!\left(\St(p,m)\times \St(p,n)\right).\label{satomet}
\end{align}
As in the previous example on $\St(p,n)$, the QR retraction on $\St(p,m)\times \St(p,n)$ is defined by
\begin{equation}
R_{(U,V)}(\xi,\eta):=\left(\qf(U+\xi),\qf(V+\eta)\right),\qquad (\xi,\eta)\in T_{(U,V)}\!\left(\St(p,m)\times \St(p,n)\right) \label{satoQRret}
\end{equation}
for $(U,V)\in \St(p,m)\times \St(p,n)$.

\begin{Prop}
The inequality \eqref{lip} holds for the objective function \eqref{satoobj} on $M=\St(p,m)\times\St(p,n)$, where $M$ is endowed with the natural induced metric \eqref{satomet} and with the QR retraction \eqref{satoQRret}.
\end{Prop}
\begin{proof}
We shall show that
\begin{equation}
\left|\frac{d^2}{dt^2}\left(F\circ R_{(U,V)}\right)(t(\xi,\eta))\right|\le L\label{appeq2}
\end{equation}
for $(\xi,\eta)\in T_{(U,V)}\!\left(\St(p,m)\times\St(p,n)\right)\ \text{with}\ \norm{(\xi,\eta)}_{(U,V)}=1,\ (U,V)\in \St(p,m)\times \St(p,n),\ t\ge 0$.
Put $Q(t)=\qf(U+t\xi),\ S(t)=\qf(V+t\eta)$.
Let $q_k(t)$ and $s_k(t)$ denote the $k$-th column vectors of $Q(t)$ and $S(t)$, respectively.
From Prop.~\ref{prop_app1} and its course of the proof, there exist vector-valued polynomials $g_k(t)$ and $h_k(t)$ such that
\begin{equation}
q_k(t)=\frac{g_k(t)}{\norm{g_k(t)}},\ s_k(t)=\frac{h_k(t)}{\norm{h_k(t)}}.
\end{equation}
Let
\begin{equation}
H(U,V,\xi,\eta,t)=\frac{d^2}{dt^2}\left(F\circ R_{(U,V)}\right)\left(t(\xi,\eta)\right).
\end{equation}
Then we have
\begin{equation}
H(U,V,\xi,\eta,t)=\sum_{k=1}^p \mu_k\frac{d^2}{dt^2}\left(q_k(t)^TAs_k(t)\right).
\end{equation}
Since $q_k(t)^TAs_k(t)=g_k(t)^TAh_k(t)/(\norm{g_k(t)}\norm{h_k(t)})$, by the same reasoning as that for $h(X,\xi,t)$ in Prop.~\ref{prop_app1}, we have
\begin{equation}
\lim_{t\to\infty}H(U,V,\xi,\eta,t)=0,
\end{equation}
so that $H(U,V,\xi,\eta,t)$ is bounded with respect to $t\ge 0$.
Further, $H(U,V,\xi,\eta,t)$ is continuous with respect to $(U,V,\xi,\eta)$ on the compact set\\ $\left\{(U,V,\xi,\eta)\in T \left(\St(p,m)\times\St(p,n)\right)\,|\,\norm{(\xi,\eta)}_{(U,V)}=1\right\}$.
Hence $H(U,V,\xi,\eta,t)$ is bounded on the whole domain.
This completes the proof.
\end{proof}

A remark similar to Rem.~\ref{remark_app1} can be made on the metric to be endowed with on $\St(p,m)\times \St(p,n)$.
The validity of \eqref{lip} is independent of the choice of a metric.

Returning to the case of a general Riemannian manifold $M$, we make a further comment on \eqref{lip}.
We are interested in the range of $t\ge 0$.
Assume that $M$ is compact and $f$ is smooth.
A smooth function on a compact set is Lipschitz continuously differentiable.
However, the set $\{(x,\eta,t)\in TM\times \mathbb R\,|\,\norm{\eta}_x=1, t\ge 0\}$
is not compact even though $M$ is compact.
Therefore, it is not so clear that the inequality \eqref{lip} holds in general.
We here note that the inequality \eqref{lip} is used in the form
\begin{equation}
\D (f\circ R_{x_k})(\alpha_k\eta_k)[\eta_k]-\D (f\circ R_{x_k})(0)[\eta_k]\le \alpha_k L\norm{\eta_k}_{x_k}^2\label{appendixlip3}
\end{equation}
for the proof of Thm.~\ref{zoutendijk}.
A question then arises as to under what condition the inequality \eqref{appendixlip3} holds.
If it is ensured that there exists a constant $m>0$ such that $\alpha_k\norm{\eta_k}_{x_k}\le m$ for all $k$, then we can prove \eqref{appendixlip3}.
Indeed, in order to prove \eqref{appendixlip3} in such a case, the range of $t$ in \eqref{lip} can be restricted to $0\le t\le m$, and the inequality we need to prove as a counterpart to \eqref{lip} is written as
\begin{equation}
\abs{\D(f\circ R_x)(t\eta)[\eta]-\D(f\circ R_x)(0)[\eta]}\le Lt, \quad \eta\in T_xM\ {\rm with}\ \norm{\eta}_x=1,\ x\in M,\ 0\le t \le m.\label{appendixlip4}
\end{equation}
In order that \eqref{appendixlip4} hold, it is sufficient that there exists a constant $L>0$ satisfying
\begin{equation}
\left|\frac{d^2}{dt^2}\left(f\circ R_x\right)(t\eta)\right|\le L,\qquad \eta\in T_x\!M\ \text{with}\ \norm{\eta}_x=1,\ x\in M,\ 0 \le t\le m.\label{appendixlip5}
\end{equation}
Since the left-hand side of the inequality \eqref{appendixlip5} is continuous with respect to $t$ on a compact set $\left\{t\in\mathbb R\,|\, 0\le t\le m\right\}$, there exists $L_{x,\eta}$ for each $(x,\eta)\in \mathcal{M}$ such that \eqref{appendixlip5} with $L=L_{x,\eta}$ holds, where $\mathcal{M}=\{(x,\eta)\in TM \,|\, \norm{\eta}_x=1\}$.
The compactness of the set $\mathcal{M}$ ensures the existence of $L:=\sup_{(x,\eta)\in \mathcal{M}} L_{x,\eta}$ and the $L$ thus defined satisfies \eqref{appendixlip5}.
\end{appendix}

\section*{Acknowledgements}
The authors would like to thank the anonymous referees for providing them with valuable comments that helped them to significantly brush up the paper.
The first author appreciates the JSPS Research Fellowship for Young Scientists.

\end{document}